\documentclass[eqno] {amsart}
\usepackage{amscd}
\usepackage{amssymb}
\usepackage{cite}
\usepackage{amsmath}
\usepackage[usenames, dvipsnames]{color}
\usepackage{amssymb}
\usepackage{bookmark}

\bookmarksetup{level=section}

\def\w*lim{\mathop{\mbox{\textup{w*-lim}}}}

\newtheorem{theorem}{\sc Theorem}[section]

\newtheorem{corollary}[theorem]{\sc Corollary}
\newtheorem{proposition}[theorem]{\sc Proposition}
\newtheorem{remark}[theorem]{\sc Remark}
\newtheorem{definition}[theorem]{\sc Definition}

\newcounter{cnt1}
\newcounter{cnt2}
\newcounter{cnt3}
\newcounter{cnt4}
\newcommand{\blr}{\begin{list}{$($\roman{cnt1}$)$} {\usecounter{cnt1}
 \setlength{\topsep}{0pt} \setlength{\itemsep}{0pt}}}
\newcommand{\blR}{\begin{list}{\Roman{cnt4}.\ } {\usecounter{cnt4}
 \setlength{\topsep}{0pt} \setlength{\itemsep}{0pt}}}
\newcommand{\bla}{\begin{list}{$($\alph{cnt2}$)$} {\usecounter{cnt2}
 \setlength{\topsep}{0pt} \setlength{\itemsep}{0pt}}}
\newcommand{\bln}{\begin{list}{$($\arabic{cnt3}$)$} {\usecounter{cnt3}
 \setlength{\topsep}{0pt} \setlength{\itemsep}{0pt}}}
\newcommand{\el}{\end{list}}

\newcommand{\cA}{{\mathcal A}}

\newcommand{\cF}{{\mathcal F}}

\newcommand{\cH}{{\mathcal H}}

\newcommand{\cK}{{\mathcal K}}

\newcommand{\cM}{{\mathcal M}}

\newcommand{\cO}{{\mathcal O}}
\newcommand{\cP}{{\mathcal P}}

\newcommand{\supp}{{\mbox{supp}}}
\newcommand{\orb}{{\mbox{Orb}}}

\newcommand{\cKO}{\cK\cO}

\newcommand{\nn}[1]{{\left\vert\kern-0.25ex\left\vert\kern-0.25ex\left\vert #1
    \right\vert\kern-0.25ex\right\vert\kern-0.25ex\right\vert}}

\begin{document}

\title[Interpolation between $L_0(\cM,\tau)$ and $L_\infty(\cM,\tau)$]{Interpolation between $L_0(\cM,\tau)$ and $L_\infty(\cM,\tau)$}

\author[J. Huang]{J. Huang}
\address[Jinghao Huang]{School of Mathematics and Statistics, University of New South Wales, Kensington, 2052, NSW, Australia  \emph{E-mail~:} {\tt
jinghao.huang@unsw.edu.au}}


\author[F. Sukochev]{F. Sukochev}
\address[Fedor Sukochev]{School of Mathematics and Statistics, University of New South Wales, Kensington, 2052, NSW, Australia  \emph{E-mail~:} {\tt f.sukochev@unsw.edu.au}
}

\subjclass[2010]{46L10, 46E30, 47A57. \hfill Version~: \today.}


\begin{abstract}
Let $\cM$ be a semifinite von Neumann algebra with a faithful semifinite normal trace $\tau$.
We show   that  the symmetrically  $\Delta$-normed operator space $E(\cM,\tau)$ corresponding to an arbitrary symmetrically $\Delta$-normed function space $E(0,\infty)$ is an interpolation  space between $L_0(\cM,\tau)$ and $\cM$,
which is in contrast with the classical result that there exist symmetric operator spaces $E(\cM,\tau)$  which are not interpolation spaces between $L_1(\cM,\tau)$ and $\cM$.
Besides, we show that the $\cK$-functional of  every  $X\in L_0(\cM,\tau)+\cM$  coincides with   the $\cK$-functional of its generalized singular value function $\mu(X)$.
Several  applications are given, e.g., it is shown that the pair $(L_0(\cM,\tau),\cM)$ is $\cK$-monotone when $\cM$ is a  non-atomic finite factor.
\end{abstract}

\maketitle


\section{Introduction}

Recall
the Calkin correspondence between symmetrically  $\Delta$-normed operator spaces and symmetrically  $\Delta$-normed function spaces introduced in \cite{HLS}.
Let $E(0,\infty)$ be an arbitrary  symmetrically $\Delta$-normed function space equipped with a $\Delta$-norm $\|\cdot\|_E$ (see Section \ref{prel}) and let $\mathcal{M}$ be an arbitrary semifinite von Neumann algebra equipped with a faithful normal semifinite trace $\tau$.
Then,
\begin{equation*}
E(\mathcal{M},\tau):=\{X\in S(\mathcal{M},\tau):\ \mu(X)\in E(0,\infty)\},
\ \|X\|_{E(\mathcal{M},\tau)}:=\|\mu(X)\|_E
\end{equation*}
 is a symmetrically  $\Delta$-normed operator space, where $S(\cM,\tau)$ is the space of all $\tau$-measurable operators affiliated with $\cM$.
Moreover, if $E(0,\infty)$ is complete, then $E(\cM,\tau)$ is also complete (see \cite[Theorem 3.8]{HLS}, see also \cite{Kalton_S, Sukochev} for the Banach case and the quasi-Banach case).
For brevity, we omit below the term ``operator'' and refer just to symmetrically  $\Delta$-normed spaces.

Let $A(\cM,\tau)$ and $B(\cM,\tau)$ be two symmetrically  $\Delta$-normed spaces.
A symmetrically  $\Delta$-normed  space $E(\cM,\tau)$ is said to be \emph{intermediate} for   $A(\cM,\tau)$ and $B(\cM,\tau)$ if the continuous embeddings
$$A(\cM,\tau)\cap B(\cM,\tau) \subset E(\cM,\tau) \subset A(\cM,\tau)+ B(\cM,\tau)$$
hold.
Let   $E(\cM,\tau)$ be  a symmetrically  $\Delta$-normed space  intermediate between $A(\cM,\tau)$ and $B(\cM,\tau)$.
If every linear operator on $A(\cM,\tau)+ B(\cM,\tau)$
which is bounded from $A(\cM,\tau)$ to $A(\cM,\tau)$ and $B(\cM,\tau)$ to $B(\cM,\tau)$ is also a bounded operator from  $E(\cM,\tau)$ to $E(\cM,\tau)$, then
$E(\cM,\tau)$ is called an \emph{interpolation} space between the spaces   $A(\cM,\tau) $ and $ B(\cM,\tau)$.

Interpolation function spaces have been widely investigated (see e.g. \cite{AK,Astashkin,HM, KPS, Rotfeld, Bergh_L,AM,Mali1992}) since Mityagin \cite{Mityagin} and Calder\'{o}n \cite{Calderon} gave characterizations of the class of all interpolation spaces with respect to $(L_1(0,\infty),L_\infty(0,\infty))$ (see also \cite{DDS2014,DDP} for results in the noncommutative setting).
Among several real interpolation methods,
the K-method of interpolation linked to  the so-called $\cK$-functional is very important (we refer \cite{MP1991,Mali1992,Astashkin,STZ} for applications of $\cK$-functionals in different areas).
Calculating  the $\cK$-functionals for a given couple of spaces is very \mbox{important}
 in the K-method \cite{Mali}.
In \cite{HM}, the $\cK$-functionals for the couple $L_0(0,\infty)$ and $L_\infty(0,\infty)$ are obtained.
We give a description of the $\cK$-functional  of every element $X\in (L_0+L_\infty)(\cM,\tau)$ in terms of singular value function as well as in terms of its distribution function,
showing that
the $\cK$-functional of $X$ coincides with the $\cK$-functional of its generalized singular value function $\mu(X)$ (see Section \ref{prel}), which extends  \cite[Proposition 3]{HM}.

It is well-known (see e.g. \cite{Mityagin,Calderon}, see also \cite{KPS,DPS,DDP}) that the  operator space $E(\cM,\tau)$ corresponding to a  fully symmetric (see e.g Section \ref{imbedding}) function space $E(0,\infty)$ is an  interpolation space between $L_1(\cM,\tau)$ and $\cM$.
In particular,
there exist symmetric normed spaces  which are not interpolation spaces between $L_1(\cM,\tau)$ and $L_\infty(\cM,\tau)$ (see \cite[Chapter II, $\S$ 4.2 and $\S$  5.7]{KPS}).
However, in this paper, it is shown that if we consider $L_0(\cM,\tau)$ (the set of all $\tau$-measurable operators having finite-trace support) instead of $L_1(\cM,\tau)$,
 then
  the operator space $E(\cM,\tau)$  corresponding to  an arbitrary symmetrically $\Delta$-normed function space $E(0,\infty)$ is necessarily an  interpolation space between $L_0(\cM,\tau)$ and $\cM$,
which is a noncommutative version of  results in  \cite{HM} (see also \cite{Astashkin}).

As an  application of the previous result,
we describe the
 orbits and $\cK$-orbits for an arbitrary $A \in S(\cM,\tau)$ in the last section.
 It is shown that   the unit balls of $\cK$-orbits do not coincide with the unit balls of orbits in the pair $(L_0(\cM,\tau),\cM)$, which generalises \cite[Theorem 4]{Astashkin}.
 In \cite{Astashkin},
 it is asserted that  the commutative pair $(L_0(0,\infty),L_\infty(0,\infty))$ is not $\cK$-monotone, that is,
$\cK$-orbits do not necessarily coincide with orbits in the pair $(L_0(0,\infty),L_\infty(0,\infty))$.
However, it is known that this assertion is incorrect and  $(L_0(0,\infty),L_\infty(0,\infty))$ is indeed $\cK$-monotone (see e.g. Section \ref{O}).
A non-commutative version of this result is established, that is,  the  pair $(L_0(\cM,\tau),\cM)$ is $\cK$-monotone  in the setting when $\cM$ is a non-atomic finite factor.
We would like to thank Professor Astashkin for providing us with the proof for the commutative pair $(L_0(0,\infty),L_\infty(0,\infty))$.

\section{Preliminaries}\label{prel}

\subsection{Generalized singular value functions}$\\$
In what follows,  $\cH$ is a Hilbert space and $B(\cH)$ is the
$*$-algebra of all bounded linear operators on $\cH$, and
$\mathbf{1}$ is the identity operator on $\cH$.
Let $\mathcal{M}$ be
a von Neumann algebra on $\cH$.
For details on von Neumann algebra
theory, the reader is referred to e.g. \cite{Dixmier}, \cite{KR1, KR2}
or \cite{Tak}. General facts concerning measurable operators may
be found in \cite{Nelson}, \cite{Se} (see also \cite[Chapter
IX]{Ta2} and the forthcoming book \cite{DPS}).
For convenience of the reader, some of the basic
definitions are recalled.

A linear operator $X:\mathfrak{D}\left( X\right) \rightarrow \cH $,
where the domain $\mathfrak{D}\left( X\right) $ of $X$ is a linear
subspace of $\cH$, is said to be {\it affiliated} with $\mathcal{M}$
if $YX\subseteq XY$ for all $Y\in \mathcal{M}^{\prime }$, where $\mathcal{M}^{\prime }$ is the commutant of $\mathcal{M}$. A linear
operator $X:\mathfrak{D}\left( X\right) \rightarrow \cH $ is termed
{\it measurable} with respect to $\mathcal{M}$ if $X$ is closed,
densely defined, affiliated with $\mathcal{M}$ and there exists a
sequence $\left\{ P_n\right\}_{n=1}^{\infty}$ in the logic of all
projections of $\mathcal{M}$, $\cP\left(\mathcal{M}\right)$, such
that $P_n\uparrow \mathbf{1}$, $P_n(\cH)\subseteq\mathfrak{D}\left(X\right) $
and $\mathbf{1}-P_n$ is a finite projection (with respect to $\mathcal{M}$)
for all $n$. It should be noted that the condition $P_{n}\left(
\cH\right) \subseteq \mathfrak{D}\left( X\right) $ implies that
$XP_{n}\in \mathcal{M}$. The collection of all measurable
operators with respect to $\mathcal{M}$ is denoted by $S\left(
\mathcal{M} \right) $, which is a unital $\ast $-algebra
with respect to strong sums and products (denoted simply by $X+Y$ and $XY$ for all $X,Y\in S\left( \mathcal{M%
}\right) $).

Let $X$ be a self-adjoint operator affiliated with $\mathcal{M}$.
We denote its spectral measure by $\{E^X\}$. It is well known that if
$X$ is a closed operator affiliated with $\mathcal{M}$ with the
polar decomposition $X = U|X|$, then $U\in\mathcal{M}$ and $E\in
\mathcal{M}$ for all projections $E\in \{E^{|X|}\}$. Moreover,
$X\in S(\mathcal{M})$ if and only if $X$ is closed, densely
defined, affiliated with $\mathcal{M}$ and $E^{|X|}(\lambda,
\infty)$ is a finite projection for some $\lambda> 0$. It follows
immediately that in the case when $\mathcal{M}$ is a von Neumann
algebra of type $III$ or a type $I$ factor, we have
$S(\mathcal{M})= \mathcal{M}$. For type $II$ von Neumann algebras,
this is no longer true. From now on, let $\mathcal{M}$ be a
semifinite von Neumann algebra equipped with a faithful normal
semifinite trace $\tau$.

An operator $X\in S\left( \mathcal{M}\right) $ is called $\tau$-measurable if there exists a sequence
$\left\{P_n\right\}_{n=1}^{\infty}$ in $P\left(\mathcal{M}\right)$ such that
$P_n\uparrow \mathbf{1},$ $P_n\left(\cH\right)\subseteq \mathfrak{D}\left(X\right)$ and
$\tau(\mathbf{1}-P_n)<\infty $ for all $n$.
The collection of all $\tau $-measurable
operators is a unital $\ast $-subalgebra of $S\left(
\mathcal{M}\right) $ denoted by $S\left( \mathcal{M}, \tau\right)
$. It is well known that a linear operator $X$ belongs to $S\left(
\mathcal{M}, \tau\right) $ if and only if $X\in S(\mathcal{M})$
and there exists $\lambda>0$ such that $\tau(E^{|X|}(\lambda,
\infty))<\infty$.
Alternatively, an unbounded operator $X$
affiliated with $\mathcal{M}$ is  $\tau$-measurable (see
\cite{FK}) if and only if
$$\tau\left(E^{|X|}\bigl(\frac1n,\infty\bigr)\right)=o(1),\quad n\to\infty.$$

\begin{definition}\label{mu}
Let a semifinite von Neumann  algebra $\mathcal M$ be equipped
with a faithful normal semi-finite trace $\tau$ and let $X\in
S(\mathcal{M},\tau)$. The generalized singular value function $\mu(X):t\rightarrow \mu(t;X)$ of
the operator $X$ is defined by setting
$$
\mu(s;X)
=
\inf\{\|XP\|_\infty:\ P=P^*\in\mathcal{M}\mbox{ is a projection,}\ \tau(\mathbf{1}-P)\leq s\}.
$$
\end{definition}
An equivalent definition in terms of the
distribution function of the operator $X$ is the following. For every self-adjoint
operator $X\in S(\mathcal{M},\tau),$ setting
$$d_X(t)=\tau(E^{X}(t,\infty)),\quad t>0,$$
we have (see e.g. \cite{FK})
\begin{align}\label{dis}
\mu(t; X)=\inf\{s\geq0:\ d_{|X|}(s)\leq t\}.\end{align}

Consider the algebra $\mathcal{M}=L^\infty(0,\infty)$ of all
Lebesgue measurable essentially bounded functions on $(0,\infty)$.
Algebra $\mathcal{M}$ can be seen as an abelian von Neumann
algebra acting via multiplication on the Hilbert space
$\mathcal{H}=L^2(0,\infty)$, with the trace given by integration
with respect to Lebesgue measure $m.$
It is easy to see that the
algebra of all $\tau$-measurable operators
affiliated with $\mathcal{M}$ can be identified with
the subalgebra $S(0,\infty)$ of the algebra of Lebesgue measurable functions which consists of all functions $x$ such that
$m(\{|x|>s\})$ is finite for some $s>0$. It should also be pointed out that the
generalized singular value function $\mu(x)$ is precisely the
decreasing rearrangement $\mu(x)$ of the function $x$ (see e.g. \cite{KPS}) defined by
$$\mu(t;x)=\inf\{s\geq0:\ m(\{|x|\geq s\})\leq t\}.$$
The two-sided ideal $\cF(\cM,\tau)$ in $\cM$ consisting of all elements of $\tau$-finite rank is defined by setting
$$\cF(\cM,\tau) =\{X \in \cM: \tau(r(X)) <\infty \} =\{X \in \cM: \tau(s(X)) <\infty \}. $$

For convenience of the reader we also recall the definition of the measure topology $t_\tau$ on the algebra $S(\cM,\tau)$.
For every $\varepsilon,\delta>0,$ we define the set $$V(\varepsilon,\delta)=\{X\in S(\mathcal{M},\tau):\ \exists P\in P\left(\mathcal{M}\right)\mbox{ such that } \|X(\mathbf{1}-P)\|\leq\varepsilon,\ \tau(P)\leq\delta\}.$$
The topology generated by the sets $V(\varepsilon,\delta)$, $\varepsilon,\delta>0,$ is called the measure topology $t_\tau$ on $S(\cM,\tau)$ \cite{DPS, FK, Nelson}.
It is well known that the algebra $S(\cM,\tau)$ equipped with the measure topology is a complete metrizable topological algebra \cite{Nelson} (see also \cite{Muratov}).
A sequence $\{X_n\}_{n=1}^\infty\subset S(\cM,\tau)$ converges to zero with respect to measure topology $t_\tau$ if and only if $\tau\big(E^{|X_n|}(\varepsilon,\infty)\big)\to 0$ as $n\to \infty$ for all $\varepsilon>0$ \cite{DPS,DP2}.



\subsection{$\Delta$-normed spaces}
$\\$
For convenience of the reader, we recall the definition of $\Delta$-norm.
Let $\Omega$ be a linear space over the field $\mathbb{C}$.
A function $\|\cdot\|$ from $\Omega$ to $\mathbb{R}$ is a $\Delta$-norm, if for all $x,y \in \Omega$ the following properties hold:
\begin{itemize}
     \item[(1)] $\|x\| \geqslant 0$, $\|x\| = 0 \Leftrightarrow x=0$;
     \item[(2)] $\|\alpha x\| \leqslant \|x\|$ for all $|\alpha| \le1$;
     \item[(3)] $\lim _{\alpha \rightarrow 0}\|\alpha x\| = 0$;
     \item[(4)] $\|x+y\| \le C_\Omega \cdot (\|x\|+\|y\|)$ for a constant $C_\Omega\geq 1$ independent of $x,y$.
\end{itemize}
The couple $(\Omega, \left\|\cdot \right\|)$ is called a $\Delta$-normed space.
We note that the definition of a $\Delta$-norm given above is the same with that given in \cite{KPR}.
It is well-known that every $\Delta$-normed space $(\Omega, \left\| \cdot \right\|)$ is metrizable  and conversely every metrizable space can be equipped with a $\Delta$-norm  \cite{KPR}.
Note that properties $(2)$ and $(4)$ of a $\Delta$-norm imply that for any $\alpha\in\mathbb{C}$, there exists a constant $M$ such that $\|\alpha x\|\leq M\|x\|,\, x\in \Omega$, in particular, if $\|x_n\|\to 0, \{x_n\}_{n=1}^\infty\subset \Omega$, then $\|\alpha x_n\|\to 0$.

%

Let $E(0,\infty)$  be a space of real-valued Lebesgue measurable
functions on  $(0,\infty)$ (with identification
$m$-a.e.), equipped with a $\Delta$-norm $\left\| \cdot \right\|_E$.
The space $E(0,\infty)$ is said to be {\it
absolutely solid} if $x\in E(0,\infty)$ and $|y|\leq |x|$, $y\in S(0,\infty)$
implies that $y\in E(0,\infty)$ and $\|y\|_E\leq\|x\|_E.$
An absolutely solid space $E(0,\infty)\subseteq S(0,\infty)$ is said to be {\it
symmetric} if for every $x\in E(0,\infty)$ and every $y\in S(0,\infty)$,
 the assumption
$\mu(y)=\mu(x)$ implies that $y\in E(0,\infty)$ and $\|y\|_E=\|x\|_E$ (see e.g.
\cite{KPS}).

We now come to the definition of the main object of this paper.
\begin{definition}\label{opspace}
Let a semifinite von Neumann  algebra $\mathcal M$ be equipped
with a faithful normal semi-finite trace $\tau$.
Let $\mathcal{E}$ be a linear subset in $S({\mathcal{M}, \tau})$
equipped with a $\Delta$-norm $\left\| \cdot \right\|_{\mathcal{E}}$. We say that
$\mathcal{E}$ is a \textit{symmetrically $\Delta$-normed operator space}  if $X\in
\mathcal{E}$ and every $Y\in S({\mathcal{M}, \tau})$ the
assumption $\mu(Y)\leq \mu(X)$ implies that $Y\in \mathcal{E}$ and
$\|Y\|_\mathcal{E}\leq \|X\|_\mathcal{E}$.
\end{definition}

One should note that a symmetrically $\Delta$-normed space $E(\cM,\tau)$ does not necessarily satisfy
$$\|AXB\|_E \le \|A\|_\infty \|B\|_\infty \|X\|_E, ~A, B\in \cM,~ X\in E(\cM,\tau).$$


It is clear that in the special case, when $\cM=L_\infty(0,1)$, or $\cM=L_\infty(0,\infty)$, or $\cM=l_\infty$, the definition of symmetrically $\Delta$-normed operator spaces coincides with  definition of the symmetric function (or sequence) spaces.
In the case, when $\mathcal{M}=B(H)$ and $\tau$
is a standard trace ${\rm Tr}$, we shall call a symmetrically $\Delta$-normed operator
space introduced in Definition \ref{opspace} a symmetrically $\Delta$-normed operator ideal (for the symmetrically normed ideals we refer to \cite{GK1, GK2,
Simon}).

As mentioned before, the operator space $E(\cM,\tau)$ defined by
\begin{equation*}
E(\mathcal{M},\tau):=\{X\in S(\mathcal{M},\tau):\ \mu(X)\in E(0,\infty)\},
\ \|X\|_{E(\mathcal{M},\tau)}:=\|\mu(X)\|_E
\end{equation*}
 is a complete symmetrically   $\Delta$-normed operator space  whenever the symmetrically  $\Delta$-normed function space $E(0,\infty)$ equipped with a $\Delta$-norm $\|\cdot\|_E$  is    complete  \cite{HLS}.

\section{$(L_0+L_\infty)(\cM,\tau) = S(\cM,\tau)$}
By $L_0(0,\infty)$ we denote the space of all measurable functions on $(0,\infty)$ whose support has finite measure. This space is endowed with the group-norm \cite{HM}
$$\|f\|_{L_0} = m(\supp (f)), $$
where $\supp (f) = \{t\in (0,\infty):f(t)\ne 0\}$.
It is clearly that the corresponding operator space $L_0(\cM,\tau)$ is the subspace of $S(\cM,\tau)$ which consists of all operators $X$ such that $\tau(s(X))<\infty$.
It is easy to see that  $(L_0+L_\infty) (\cM,\tau)$ coincides with $S(\cM,\tau)$.
For the sake of completeness, we present a brief proof below.
\begin{proposition}\label{prop:0infty}
$(L_0+L_\infty) (\cM,\tau) =S(\cM,\tau)$.
\end{proposition}
\begin{proof}
Since  $ (L_0+L_\infty) (\cM,\tau) \subset S(\cM,\tau)$, it suffices to prove $ S(\cM,\tau)\subset (L_0+L_\infty) (\cM,\tau) $.
For any operator $X\in S(\cM,\tau)$, there exists $\lambda >0$ such that $t:= \tau(E^{|X|}(\lambda,\infty))<\infty$.
By  \cite[Chapter III, Eq. (4)]{DPS}, we have
\begin{align*}
d_{\mu(X)}(\lambda)= \tau(E^{|X|}(\lambda,\infty)) =t .
\end{align*}
This  together with \cite[Proposition 3]{HM} implies that  $\mu(X)\in (L_0+L_\infty) (0,\infty)$.
Thus, $X\in (L_0+L_\infty) (\cM,\tau)$.
\end{proof}

In \cite[Proposition 3]{HM}, the description of the $\cK$-functional (which is a $\Delta$-norm on $S(0,\infty)$) $K_u(f)=\inf\{\|g\|_0 +u\|h\|_\infty : f=g+h, g\in L_0(0,\infty), h\in L_\infty(0,\infty)\}$, $u>0$,
of any  $f\in S(0,\infty)$  is given in terms of its distribution function and its singular value function.
That is, for every $f\in S(0,\infty)$, we have
\begin{align}\label{ineq:K}
K_u(f)= \inf_{s>0} [su + d_{\mu(f)}(s)]= \inf_{t>0} [ t+ u \mu(t;f)].
\end{align}
Similarly, for every $X\in S(\cM,\tau)$, the $\cK$-functional is defined by
$$K_u(X) : =\inf\{\|G\|_0 +u\|H\|_\infty : X=G+H,G\in L_0(\cM,\tau), H\in L_\infty(\cM,\tau)\}, u>0. $$
In particular, we define a $\Delta$-norm (see e.g. Remark \ref{remark:tri}) on $S(\cM,\tau)$ by
$$\|X\|_S:= K_1 (X)= \inf\{\|G\|_0 +\|H\|_\infty : X=G+H,G\in L_0(\cM,\tau), H\in L_\infty(\cM,\tau)\}$$
 for any $X\in S(\cM,\tau)$.
The following result complements an earlier result from \cite{DDP} for the pair $(L_1(\cM,\tau),L_\infty(\cM,\tau))$.
\begin{proposition}\label{prop:3.3}
For every $X\in S(\cM,\tau)$, we have
$$K_u(X)  =K_u(\mu(X)), ~u>0. $$
In particular, $\|X\|_S= \inf_{s>0} [s + d_{\mu(X)}(s)]= \inf_{t>0} [ t+ \mu(t;X)]$.
Moreover, $S(\cM,\tau)$ is a complete $\Delta$-normed space with respect to the $\Delta$-norm $K_u$ for every $u>0$.
\end{proposition}
\begin{proof}
Firstly, for every $X\in S(\cM,\tau)$, we have
\begin{align*}
K_u(X)&= \inf\{\|G\|_0 +u\|H\|_\infty : X=G+H, G\in L_0(\cM,\tau), H\in L_\infty(\cM,\tau)\}\\
&\le \inf_{ s>0 } \{\| X E^{|X|}(s,\infty)\|_0 +u\|X E^{|X|}(0,s]\|_\infty  \} \\
&\le  \inf_{ s>0 } \{ d_{ |X E^{|X|}(s,\infty)|}(0  ) +us  \} \\
&= \inf_{ s>0 } \{ d_{\mu( X E^{|X|}(s,\infty))}(0  ) +us  \}  ~\quad \mbox{(by \cite[Chapter III, Eq. (4)]{DPS})}  \\
&= \inf_{ s>0 } \{ d_{\mu( X E^{|X|}(s,\infty))}(s  ) +us  \} \\
&\le  \inf_{ s>0 }\{d_{\mu(X)}(s) +us \} \\
& \stackrel{\eqref{ineq:K} }{=} K_u(\mu(X)).
\end{align*}
Conversely,
for every $G\in  L_0(\cM,\tau) $ with $t:= \|G \|_0 $ and $X-G\in L_\infty(\cM,\tau)$, by \cite[Chapter III, Proposition 2.20]{DPS},  we have
\begin{align*}
&~\quad  t+  u \mu(t;X)\\
& =  t+  u\cdot  \inf\{\|X-Y\|_\infty : Y\in L_0(\cM,\tau), \|Y\|_0\le t, X-Y\in \cM \} \\
 & \le  \|G\|_0  +  u \| X - G \|_\infty .
\end{align*}
 Hence, we obtain
 $$ \inf_{t>0}[t+  u \mu(t;X) ]\le  \inf _{\substack{G\in L_0(\cM,\tau)\\ X- G\in L_\infty(\cM,\tau)}}( \|G\|_0  +  u \| X - G \|_\infty )=K_u(X) .$$
 The fact that $S(0,\infty)$ is complete with respect with $K_u$ together with  \cite[Theorem 3.8]{HLS} implies the completeness of $S(\cM,\tau)$ with respect with $K_u$.
\end{proof}

\begin{remark}
Recall that $d_{\mu(X)}(s)=d_{|X|}(s)$ (see e.g. \cite[Chapter III, Eq. (4)]{DPS}).
For every $X\in S(\cM,\tau)$, the $\cK$-functional can be also defined by the formula
\begin{align*}
K_u(X) = \inf_{s>0} [su + d_{|X|}(s)],~ u>0.
\end{align*}
\end{remark}

\begin{remark}\label{remark:tri}
It is still unknown  whether the Calkin correspondence preserves the constant $C_E$ for an arbitrary symmetrically  $\Delta$-normed  function space $E(0,\infty)$ (see \cite{HLS,Sukochev}).
However, it is well-known that  $(S(0,\infty),K_u(\cdot))$, $u>0$, is an $F$-space (i.e., a complete $\Delta$-normed space with $C_E=1$) and
Proposition \ref{prop:3.3} implies that
 for every $u>0$,
 $K_u(\mu( \cdot))$
  is not only a $\Delta$-norm but also an $F$-norm on $S(\cM,\tau)$.
Indeed, for every $X,Y\in S(\cM,\tau)$, by Proposition \ref{prop:3.3}, we have
\begin{align*}
&~\quad K_u(\mu(X)) +K_u(\mu(Y))\\
&= K_u(X) +K_u(Y)\\
& =\inf\{\|X_1\|_0 +u\|X_2\|_\infty : X=X_1+X_2,X_1\in L_0(\cM,\tau), X_2\in L_\infty(\cM,\tau)\} \\
&~\quad +\inf\{\|Y_1\|_0 +u\|Y_2\|_\infty : Y=Y_1+Y_2,Y_1\in L_0(\cM,\tau), Y_2\in L_\infty(\cM,\tau)\}\\
& =\inf\{\|X_1\|_0 +\|Y_1\|_0 +u\|X_2\|_\infty +u\|Y_2\|_\infty : \\
&~\quad X=X_1+X_2, Y=Y_1+Y_2, X_1,Y_1\in L_0(\cM,\tau), X_2,Y_2\in L_\infty(\cM,\tau)\}\\
& \ge \inf\{\|X_1+Y_1\|_0 +u\|X_2+ Y_2\|_\infty : \\
&~\quad X=X_1+X_2, Y=Y_1+Y_2, X_1,Y_1\in L_0(\cM,\tau), X_2,Y_2\in L_\infty(\cM,\tau)\}\\
& \ge \inf\{\|Z_1\|_0 +u\|Z_2\|_\infty :  X+Y=Z_1+Z_2, Z_1\in L_0(\cM,\tau), Z_2\in L_\infty(\cM,\tau)\}\\
&= K_u(X+Y)= K_u(\mu(X+Y)),
\end{align*}
where we used the  fact that $\|X\|_0+ \|Y\|_0  = \tau(\supp(X)) +\tau(\supp(Y)) \ge \tau(\supp(X+Y))=\|X+Y \|_0$ for every $X,Y\in S(\cM,\tau)$.
\end{remark}

\section{An embedding theorem}
It is well-known that for every  symmetrically normed function space $E(0,\infty)$, the corresponding   operator space $E(\cM,\tau)$ is  symmetrically normed \cite{Kalton_S} and is an intermediate space for the noncommutative pair $(L_1(\cM,\tau),\cM)$ \cite{DP2,DPS}.
In this section, we prove an analogue for the $\Delta$-normed case, that is,  every operator space $E(\cM,\tau)$ corresponding to  a  $\Delta$-normed  function space $E(0,\infty)$  is an intermediate space for the noncommutative pair $(L_0(\cM,\tau),\cM)$.

Before we proceed to the proof of the embedding theorem, we show that the topology given by $\|\cdot\|_S$ is equivalent with the measure topology.
\begin{proposition}\label{prop:SandM}
Let $\{X_n\}$ be a sequence in $S(\cM,\tau)$.
Then, $\|X_n\|_S\rightarrow 0$ if and only if $X_n \rightarrow _{t_\tau} 0$.
\end{proposition}
\begin{proof}
By \cite[Lemma 2.4]{HLS}, it suffices to show that $\|X_n\|_S\rightarrow 0$ whenever  $X_n \rightarrow _{t_\tau} 0$.
By \cite[Chapter II, Proposition 5.7]{DPS}, we have $\tau(E^{|X_n|} (\varepsilon,\infty))\rightarrow_n 0$ for every $\varepsilon>0$.
By Proposition \ref{prop:3.3}, we have  $\|X_n\|_S \le \varepsilon +\tau(E^{|X_n|} (\varepsilon,\infty))$, which completes the proof.
\end{proof}

Notice that the two-sided ideal $\cF(\tau)$ in $\cM$ coincides with $(L_0\cap L_\infty) (\cM,\tau)$.
For every $X\in \cF(\tau)$,
we define the group-norm $\|X\|_\cF$
by
$$ \|X\|_\cF := \max \{\|X\|_0, \|X\|_\infty\}.$$
The following embedding theorem is the main result of this section, which extends \cite[Theorem 1]{HM} to the non-commutative case.
\begin{theorem}\label{th:embedding}
If $E(0,\infty)$ is a nontrivial symmetrically   $\Delta$-normed  function space, then
$$ \cF (\tau) \subset E(\cM,\tau)\subset S(\cM,\tau).$$
Moreover, the embeddings are continuous.
That is, $E(\cM,\tau)$ is an intermediate space between $L_0(\cM,\tau)$ and $\cM$.
\end{theorem}
\begin{proof}
Since $E(0,\infty)$ is not empty, there is a non-zero element $ x_0 \in E(0,\infty)$.
Then, there is a scalar $t>0$ such that $\mu(t;x_0)>0$.
It is clear that $\mu(t;x_0)\chi_{(0,t]} \le \mu(x_0)$, which implies that $\chi_{(0,t]} \in E(0,\infty)$.
Since $ m(\chi_{(0,t]})=  m(\chi_{(t,2t]})= \cdots = m(\chi_{((n-1)t,nt]})$  and $E(0,\infty)$ is a linear space, it follows that $\chi_{(0,nt]}\in E(0,\infty)$.
Hence, $\cF(0,\infty )\subset E(0,\infty)$.

Let  $\{X_n\}_n \subset \cF(\tau)$ be a sequence  such that $\|X_n\|_\cF \rightarrow 0$.
For every $0<\varepsilon<1$, we can find an $N$ such that for every $n\ge N$, we have $\|X_n \|_{\cF}\le \varepsilon$, that is,
$$\supp (\mu(X_n))\le \varepsilon ~ \mbox{ and } ~ \|X_n\|_\infty \le \varepsilon.$$
Hence, $\mu(X_n) \le \varepsilon \chi_{(0,\varepsilon]}\le \varepsilon \chi_{(0,1]}$ and therefore $\|X_n \|_E  \le \|\varepsilon \chi_{(0,1]}\|_E $.
By the continuity of $\Delta$-norm $\|\cdot\|_E$, we obtain that $\|X_n\|_E\rightarrow_n 0$.

Lemma 2.4 in \cite{HLS} together with   Proposition \ref{prop:SandM} implies that $E(\cM,\tau)$ is continuously embedded into $ S(\cM,\tau)$.
\end{proof}

The set of all self-adjoint elements in $E(\cM,\tau)$ is denoted by $E_h(\cM,\tau)$.
Then,  \cite[Chapter II, Proposition 6.1]{DPS} together with  Proposition \ref{prop:SandM} and Theorem \ref{th:embedding} implies the following results immediately.
\begin{corollary}\label{cor:cone}
Let  $E(0,\infty)$ is a symmetrically    $\Delta$-normed  function space.
The following statements hold.
\begin{enumerate}
                   \item The positive cone $E(\cM,\tau)^+$ is closed in $E(\cM,\tau)$ with respect to $\|\cdot\|_E$.
                   \item
                         If $\{X_n\}_{n=1}^\infty$ is a sequence in $E(\cM,\tau)$ and $X,Y\in E_h(\cM,\tau)$ are such that $\|X_n -X\|_E\rightarrow_n 0$ and $X_n\le Y$ for all $n$, then $X\le Y$.
                   \item If $\{X_n\}_{n=1}^\infty$ is an increasing  sequence in $E_h(\cM,\tau)$ and $X\in E_h(\cM,\tau)$ with $\|X_n-X\|_E \rightarrow_n 0$, then $X_n\uparrow_n X$.
                 \end{enumerate}
\end{corollary}

\section{Interpolation in the pair $(L_0(\cM,\tau),\cM)$}\label{imbedding}

Introduce the dilation operator $\sigma_s$ on $S(0,\infty)$, $s>0$, by setting
\begin{align*}
(\sigma_s(x))(t) = x \left( \frac{t}{s} \right), ~t >0.
\end{align*}
It is well-known that $\mu(X+Y)\leq \sigma_2(\mu(X)+\mu(Y))$, $X,Y\in S(\cM,\tau)$ \cite{LSZ}.
We note also that $\sigma_{2^k}x \in E(0,\infty)$ with
\begin{equation}\label{norm_sigma}
\|\sigma_{2^k}x\|_E\leq (2C_E)^k\|x\|_E
\end{equation}
for all $x\in E(0,\infty)$ and $k\in\mathbb{N}$ (see e.g. \cite{KPS}, see also \cite{HM}).

Recall that $(L_0+L_\infty)(\cM,\tau)=S(\cM,\tau)$ (see Proposition \ref{prop:0infty}).
Let $T: S (\cM,\tau)\rightarrow S(\cM,\tau)$ be a homomorphism, i.e.,
$$  T(X+Y) =TX+TY \mbox{ and }T(-X)=-TX$$ for any $X,Y \in S(\cM,\tau)$.
Let $E(0,\infty)$ be a $\Delta$-normed function space.
A homomorphism $T: E(\cM,\tau) \rightarrow E(\cM,\tau)$ is called \textbf{continuous} if
for any given $\varepsilon>0$, there exists $\delta(\varepsilon)>0$ such that $\|X\|_E <\delta(\varepsilon)$ implies that $\|TX\|_E <\varepsilon$ \cite[Chapter I, Section 4]{KPR}.
A homomorphism is called \textbf{bounded}  if
$$\|T\|_{E\rightarrow E} =\sup _{x\ne 0} \frac{\|TX\|_E}{\|X\|_E}<\infty.$$
The homomorphism $T$ is said to be bounded on the pair $(L_0(\cM,\tau),\cM)$ if $T$ is a bounded mapping from  $L_0(\cM,\tau)$ into $L_0(\cM,\tau)$ and from $\cM$ into $\cM$.
\begin{theorem}\label{th:bounded}
Let $E(0,\infty)$ be a symmetrically    $\Delta$-normed  function space and $T: S (\cM,\tau)\rightarrow S(\cM,\tau)$ be a homomorphism which is bounded on $(L_0(\cM,\tau),\cM)$ with
$$ \| TX\|_0 \le M_0 \|X\|_0, \  \forall X\in L_0(\cM,\tau),$$
$$ \|TX\|_\infty \le M_1 \|X\|_\infty, \  \forall X\in \cM $$
for some constants $M_0,M_1>0$.
Then, $T$ maps $E(\cM,\tau)$ into itself and
 \begin{align}\label{ineq:M1}
\mu(M_0t; TX)  \le \mu(t;M_1X),~ X\in E(\cM,\tau).
 \end{align}
\end{theorem}
\begin{proof}
For any $P\in \cP(\cM)$ with $\tau(\mathbf{1}-P)<t$, we have $ \|T X({\bf 1}-P)\|_0 \le M_0\|X( \mathbf{1}-P)\|_0 \le M_0 t$.
By \cite[Theorem 2.3.13]{LSZ}, for every $t>0$, we have
\begin{align*}
\mu(M_0t;TX)
&=
\inf\{ \|TX-B\|_\infty : B\in S(\cM,\tau), \ \|B\|_0 \le M_0t\}\\
&\le  \|TX - TX({\bf 1}-P)\|_\infty =  \| TX P\|_\infty \le M_1 \|XP\|_\infty.
\end{align*}
By Definition \ref{mu}, we have
\begin{align*}
\mu(M_0t; TX)  \le M_1 \mu(t; X).
\end{align*}
This implies that  $\sigma_{1/M_0}\mu( TX) \in E(0,\infty)$ and therefore, by appealing to \eqref{norm_sigma}, we conclude that   $\mu( TX) =\sigma_{M_0}(\sigma_{1/M_0}\mu( TX) )\in E(0,\infty)$.
\end{proof}

If $X,Y\in S(\cM,\tau)$, then $X$ is said to be submajorized by $Y$, denoted by $X\prec\prec Y$, if \begin{align*}
\int_{0}^{t} \mu(s;X) ds \le \int_{0}^{t} \mu(s;Y) ds , ~t\ge 0.
\end{align*}
A linear subspace $E$ of $S(\cM,\tau)$ equipped with a complete norm $\|\cdot\|_E$, is called \emph{fully symmetric space} (of $\tau$-measurable operators) if $X\in S(\cM,\tau)$, $Y \in E$ and $X\prec\prec Y$ imply that $X\in E$ and $\|X\|_E \le \|Y\|_E$ \cite{DP2,DPS,LSZ}.

For a symmetric normed function space $E(0,\infty)$,
by \cite[Theorem 10.13]{DPS} (see also \cite{DDP} and \cite{KPS}), $E(\cM,\tau)$ is an interpolation space between  $L_1(\cM,\tau)$ and $\cM$ if  $ E(0,\infty)$ is fully symmetric.
In particular, one can find symmetric normed spaces $E(\cM,\tau)$ which are not interpolation spaces between $L_1(\cM,\tau)$ and $L_\infty(\cM,\tau)$ \cite[Chapter II, $\S$  5.7]{KPS}.
However,
for an  arbitrary symmetrically    $\Delta$-normed  function space $E(0,\infty)$, $E(\cM,\tau)$ is, in fact, an interpolation space between  $L_0(\cM,\tau)$ and $\cM$.

\begin{corollary}\label{cor:inter}
Let $E(0,\infty)$ be a symmetrically    $\Delta$-normed  function space and $T: S (\cM,\tau)\rightarrow S(\cM,\tau)$ be a homomorphism which is bounded on $(L_0(\cM,\tau),\cM)$ with
$$ \| TX\|_0 \le M_0 \|X\|_0, \  \forall X\in L_0(\cM,\tau),$$
$$ \|TX\|_\infty \le M_1 \|X\|_\infty, \  \forall X\in \cM$$
for some constants $M_0,M_1>0$.
Then, $T$ is a bounded homomorphism from $E(\cM,\tau)$ into itself.
In particular, $\sup_{X\in S(\cM,\tau)}\frac{\|TX\|_{S}}{\|X\|_{S}} <\infty$.
\end{corollary}
\begin{proof}
By Theorem \ref{th:bounded}, for any $X\in E(\cM,\tau)$, we have $TX\in E(\cM,\tau)$ with
\begin{align}\label{ineq:M}
\mu(M_0t; TX)  \le M_1 \mu(t; X).
\end{align}
Let $k\ge 0$ be an  integer such that $2^k \ge  M_0$.
Noticing that $\sigma_{2^k}(\mu(X) ) \ge \sigma_{M_0} (\mu(X))$.
By (\ref{norm_sigma}), we have
\begin{equation}\label{ineq:CE}
\|\sigma_{M_0} (\mu(X))\|_E \le \|\sigma_{2^k}\mu(X)\|_E\le (2C_E)^k\|\mu(X)\|_E
\end{equation}
for any $X\in E(\cM,\tau)$.

Then,  we get
 \begin{align*}
 \|TX\|_E  &= \|\mu(TX)\|_E  \stackrel{\eqref{ineq:CE}}{\le} (2C_E)^k \|\sigma_{1/M_0}(\mu(TX))\|_E \stackrel{\eqref{ineq:M}}{\le} (2C_E)^k\|M_1 \mu(X)\|_E\\
 &\le (2C_E)^k \|([M_1]+1)  \mu(X)\|_E \le (2C_E)^k  \sum_{i=1}^{[M_1]+1}C_E^i \|  \mu(X)\|_E,
 \end{align*}
where  $[M_1]$ is the  integer part of $ M_1$.
The proof is complete.
\end{proof}

The results in this section are applied to the  study of   orbits and $\cK$-orbits in the pair of $(L_0(\cM,\tau),\cM)$ in the next section.

\section{Orbits and $\cK$-orbits}\label{O}
For an element $X\in S (\cM,\tau)$,  the orbit $\orb(X;L_0(\cM,\tau),\cM)$ of $X$ is  the set of all $Y\in S (\cM,\tau)$ such that $Y = TX$ for some  homomorphism $T$ which is \emph{bounded} on the pair $(L_0 (\cM,\tau), \cM)$.
Furthermore, we define
$$\|Y\|_{\orb} :=\inf_{Y=TX} \|T\|_{S(\cM,\tau)},$$
where the infimum is taken over all bounded homomorphisms $T$ such that $TX=Y$ and $\|T\|_{S(\cM,\tau)} =\max \{ \|T\|_{L_0\rightarrow L_0}, \|T \|_{L_\infty \rightarrow L_\infty}\}$.

By Theorem \ref{th:bounded}, we have the following proposition, which is an analogue of \cite[Theorem 1]{Astashkin}.
\begin{proposition}\label{prop:6.1}
Let  $X\in S(\cM,\tau)$.
Then,  for every  $Y\in \orb(X; L_0(\cM,\tau),\cM )$,  we have
\begin{align}\label{ineq:C}
\mu(  t; Y)  \le \|Y\|_\orb \mu(\frac{t}{\|Y\|_\orb}; X),~ t>0.
\end{align}
\end{proposition}
\begin{proof}
Notice for every $\varepsilon>0$, we can find a $T$ such that $Y=TX$ with $\|T\|_{S(\cM,\tau)} \le \|Y\|_{\orb} +\varepsilon $.
Then, by (\ref{ineq:M1}),  we have
\begin{align*}
\mu(\|T\|_{S(\cM,\tau)} t; Y)=\mu(\|T\|_{S(\cM,\tau)} t; TX) \le \|T\|_{S(\cM,\tau)} \mu(t; X) .
\end{align*}
Hence, $\mu( (\|Y\|_{\orb} +\varepsilon) t; Y)  \le (\|Y\|_{\orb} +\varepsilon)\mu(t; X)  $ for every $t>0$.
By the right-continuity of singular value functions, we have $$\mu(\|Y\|_{\orb} t; Y)  \le \|Y\|_{\orb} \mu(t; X)  $$
for every $t>0$, which completes the proof.
\end{proof}

Let $(X_0,X_1)$ be a pair of symmetrically    $\Delta$-normed spaces.
The $\cK $-orbit of $A\in X_0+X_1$ is defined by the set $\cK\cO(A;X_0,X_1)$ of all $X \in X_0+X_1$
such that
$$\|X\|_{\cK\cO} :=\sup_{t>0} \frac{\cK(t,X; X_0,X_1)}{\cK(t, A; X_0,X_1)}<\infty,$$
where $\cK(t,Z; X_0,X_1) := \inf\{\|Z_0\|_{X_0}+t\|Z_1\|_{X_1}: Z= Z_0 +Z_1 , Z_0\in X_0, Z_1\in X_1\}$.

A pair $(X_0,X_1)$ is called \emph{$\cK$-monotone} if $\cK\cO(A;X_0,X_1) = \orb(A;X_0,X_1)$ for all $A\in X_0+X_1$. The pair $(L_1(0,\infty), L_\infty(0,\infty))$ is a classical example of a $\cK$-monotone pair (see e.g. \cite{Calderon}).
Moreover, the noncommutative pair $(L_1(\cM,\tau),L_\infty(\cM,\tau))$ is $\cK$-monotone (see \cite[Proposition 2.5, Theorem 4.7]{DDP}).


It follows from the definition  that the  unit ball of $\orb(A;X_0,X_1)$ is a subset of the unit \mbox{ball} of $\cKO(A;X_0,X_1)$.
Moreover,
\cite[Proposition 2.5 and Theorem 4.7]{DDP} imply that the the closed unit ball of $\orb(A;L_1(\cM,\tau),\cM)$ coincides with the unit ball of $\cKO(A;L_1(\cM,\tau),\cM)$.
However, it is known that the reverse inclusion may fail for certain element in  the pair $L_0(0,\infty)+ L_\infty(0,\infty)$ \cite{Astashkin}.
One of the main results of this section is a non-commutative version of \cite[Theorem 4]{Astashkin}.

By Proposition \ref{prop:3.3},
 the  $\cK $-orbit $\cK\cO(A;L_0(\cM,\tau),\cM)$ of  every $\cA\in S(\cM,\tau)$ is the set of all $X \in S(\cM,\tau)$
such that
$$\|X\|_{\cK\cO} :=\sup_{t>0} \frac{K_t(\mu(X))}{K_t(\mu(A))}<\infty.$$
\begin{theorem}If $\cM$ is a non-trivial von Neumann algebra ($\cM\ne \mathbb{C}{\bf 1}$ and  $\cM\ne 0$), then there exist $A,X\in S(\cM,\tau)$ such that
$$K_t(A) = K_t(X)$$
whereas $\mu(t;A)<\mu(t;X)$, $t\in E$, for some measurable set $E$, $m(E)>0$.
In particular, the unit ball of the $\cKO(A;L_0(\cM,\tau),\cM)$ does not coincide with the unit ball of $\orb(A;L_0(\cM,\tau),\cM)$.
\end{theorem}
\begin{proof}
Since $\cM \ne \mathbb{C}\bf 1$, there exist two $\tau$-finite projections $P_1,P_2\in \cP(\cM)$ such that $P_1\perp P_2 $.
Let $\tau_1:= \tau(P_1)>0$ and $\tau_2:=\tau(P_2)>0$.

Let $ k_1, k_2>0$ be such that
\begin{align}\label{ineq:tau_k}
k_1>k_2 >\frac{\tau_2 k_1}{\tau_1+\tau_2}.
\end{align}

Define $$X:= k_1( P_1 +P_2).$$
Then, $\mu(X) = k_1 \chi_{(0, \tau_1+\tau_2)}$.
By (\ref{ineq:K}), we have
$$ K_t(X) = \min\{tk_1,  \tau_1+\tau_2\}.$$
Then, we have
$$ K_t(X)=\left\{
\begin{aligned}
t k_1, & \qquad\qquad  t<  \frac{\tau_1+\tau_2}{k_1 }, \\
\tau_1+\tau_2   , &  \qquad\qquad t \ge \frac{\tau_1+\tau_2}{k_1 }.
\end{aligned}
\right.
$$
Define $$A:= k_1 P_1 +  k_2 P_2.$$
Then, $\mu(A) = k_1 \chi_{(0,\tau_1)} + k_2 \chi_{[\tau_1,\tau_1+\tau_2)}$.
By (\ref{ineq:K}), we have
$$ K_t(A) = \min\{tk_1,  \tau_1 + t k_2,  \tau_1+\tau_2\}.$$
However, (\ref{ineq:tau_k}) implies that there is no such a $t$ such that $ \tau_1 + t k_2 \le \min\{tk_1,  \tau_1+\tau_2\}$.
Hence, we have
$$ K_t(A)=\left\{
\begin{aligned}
t k_1, & \qquad\qquad  t<  \frac{\tau_1+\tau_2}{k_1 } ,\\
\tau_1+\tau_2   , &  \qquad\qquad t \ge \frac{\tau_1+\tau_2}{k_1 }.
\end{aligned}
\right.
$$
That is, $K_t(A)=K_t(X)$. However, it is clear that
\begin{align}\label{9}
\mu(t;A) <  \mu(t;X), ~\quad  \tau_1\le t <\tau_1+\tau_2.
\end{align}

Assume that $X$ lies in the unit ball of $ \orb(A;L_0(\cM,\tau),\cM)$.
By (\ref{ineq:C})
\begin{align*}
\mu(  t; X)  \le \|X\|_\orb \mu(\frac{t}{\|X\|_\orb}; A) \le \mu(t ; A), ~t>0,
\end{align*}
which is a contradiction with \eqref{9}.
Hence,
$X$ lies in the unit ball of $\cKO(A;L_0(\cM,\tau),\cM)$ but not in the unit ball of $\orb(A;L_0(\cM,\tau),\cM)$.
\end{proof}

In \cite{Astashkin}, it is  asserted incorrectly that the $\orb (A; L_0(0,\infty), L_\infty(0,\infty)) \ne \cKO (A;L_0(0,\infty), L_\infty(0,\infty))$.
The following proposition together \cite[Theorem 1]{Astashkin} explains why  $\orb (A; L_0(0,\infty), L_\infty(0,\infty)) = \cKO (A;L_0(0,\infty), L_\infty(0,\infty))$, i.e.,  the pair $(L_0(0,\infty),L_\infty(0,\infty))$ is $\cK$-monotone.
We would like to thank Professor Astashkin for providing the proof for the special case when $\cM=L_\infty(0,\infty)$.
\begin{proposition}\label{prop:6.5}
Let $A,X\in S(\cM,\tau)$.
Then, the following statements are equivalent.
\begin{enumerate}
  \item There exists $C >0$ such that $\mu(s;X)\le C \mu(\frac{s}{C};A)$ for every $s>0$.
  \item $\sup_{t>0} \frac{K_t(X)}{K_t(A)}<\infty$.
\end{enumerate}
\end{proposition}
\begin{proof}
(i) For every $A,X$ satisfying condition (1), we have
\begin{align*}
\sup_{t>0}  \frac{K_t(X)}{K_t(A)} &= \sup_{t>0}  \frac{\inf_s\{s +  t \mu(s;X) \}}{\inf_s\{s +  t \mu(s;A) \}}\le \sup_{t>0}  \frac{\inf_s\{s +  t C \mu(s/C ;A) \}}{\inf_s\{s +  t \mu(s;A) \}}\\
&= \sup_{t>0}  \frac{\inf_s\{C s +  t C \mu(s  ;A) \}}{\inf_s\{s +  t \mu(s;A) \}}= C <\infty,
\end{align*}
which proves the validity of condition (2).

(ii) Conversely, assume that
$$\sup_{t>0}  \frac{\inf_s\{s +  t \mu(s;X) \}}{\inf_s\{s +  t \mu(s;A) \}} \le C$$
for some $C>0$.
For every $Z\in S(\cM,\tau)$, we define
$$
M_t(Z):= \inf_s\{\max\{s,t \mu(s;Z)\}\}, ~  t>0 .$$
Clearly, we have
$$ M_t(Z)\le K_t(Z)\le 2M_t(Z),\;\;Z\in S(\cM,\tau), ~t>0 \mbox{\cite{MP1991}}.$$
Therefore,
\begin{eqnarray}\label{eq1}
M_t(X)\le 2CM_t(A),\;\;t>0.
\end{eqnarray}

Let $s\in(0,\infty)$ and
let $t:=\frac{s}{\mu(s;X)}$ (without loss of generality, we may assume that  $\mu(s;X)>0$).
Notice that  $s=t\mu(s;X)\le t\mu(s-\Delta_1 ;X)$ and $s=t\mu(s;X)\le s+\Delta_2 $ for any $\Delta_1 ,\Delta_2 >0$ with  $s-\Delta_1 > 0$.
We have,
\begin{align}\label{s=}
M_t(X)=s=t\mu(s;X).
\end{align}
Let $s_1:=M_t(A)$.
Then, we have
\begin{align}\label{s-}
t \mu(s_1^-;A )(:=t\lim_{k \uparrow s_1} \mu(k;A))\ge s_1
\end{align}
(otherwise, we have $\max\{t\mu(s_1-\varepsilon;A ),s_1-\varepsilon \}<s_1 = M_t(A)$ for some $\varepsilon>0$, which is a contradiction to the definition of $M_t (A)$.).

Since $s  \stackrel{\eqref{s=}}{=} M_t(X) \stackrel{\eqref{eq1}}{\le} 2C M_t(A) = 2Cs_1 $, it follows that
 \begin{align}\label{esi1}
 \mu(s_1^-;A)\le \mu((\frac{s}{2C})^-;A).
 \end{align}
Then, we obtain
$$
t\mu(s;X) \stackrel{\eqref{s=}}{=} M_t(X) \stackrel{\eqref{eq1}}{\le} 2CM_t(A) = 2Cs_1  \stackrel{\eqref{s-}}{\le} 2C t \mu(s_1^-;A)\stackrel{\eqref{esi1}}{\le} 2Ct \mu((\frac{s}{2C})^-;A),$$
that is,
\begin{align*}
\mu(s;X)\le 2C \mu((\frac{s}{2C})^-;A) .
\end{align*}
Since $s>0$ is arbitrary taken,  it follows that
$$\mu(s;X)\le 3C \mu(\frac{s}{3C};A) $$
for every $s>0$, which completes the proof.
\end{proof}

By the above proposition  and  \cite[Theorem 1]{Astashkin},  we obtain the  following result immediately, which implies that the commutative pair $(L_0(0,\infty), L_\infty(0,\infty))$ is indeed $\cK$-monotone.
\begin{corollary}
For every $a\in S(0,\infty)$, we have
$$\orb(a;L_0(0,\infty),L_\infty(0,\infty)) = \cKO(a;L_0(0,\infty),L_\infty(0,\infty)).$$
\end{corollary}

It is known that there cannot in general be a conditional expectation from $S(\cM,\tau)$ onto a subalgebra of $S(\cM,\tau)$  (see e.g. \cite[Appendix B]{DSZ}), which is the main obstacle in extending \cite[Theorem 1]{Astashkin} to the non-commutative case.
The following theorem is the last result of this section, giving a non-commutative version of \cite[Theorem 1]{Astashkin} in the setting of non-atomic finite factors by approaches which are completely different from those used in \cite{Astashkin}.

For the sake of convenience, we denote $L_\infty(0,\tau(s(X)))$ by $M_{\mu(X)}$.
If $\cM$ is a  non-atomic semifinite von Neumann algebra, then for every $X\in S_0(\cM,\tau)$, there exists a non-atomic commutative von Neumann subalgebra $\cM_{|X|}$ in $s(|X| )\cM s(|X|)$ and a trace-preserving $*$-isomorphism $J$ from $S(\cM_{|X|},\tau)$ onto the algebra $S(M_{\mu(X)},m)$ \cite{CKS,DDP1992}.

\begin{theorem}\label{th:6.5}
If $\cM$ is a non-atomic finite von Neumann factor with a faithful normal finite trace $\tau$, then for every $0 \ne A\in S(\cM,\tau)$,
$\orb(A;L_0(\cM,\tau),L_\infty(\cM,\tau))$ is the set of all $X\in S(\cM,\tau)$ for which there exists $C>0$ such that
\begin{align}\label{CCC}
\mu(t;X) \le C \mu(\frac{t}{C};A),~t>0.
\end{align}
In particular,
$\orb(A;L_0(\cM,\tau),L_\infty(\cM,\tau)) = \cKO(A;L_0(\cM,\tau),L_\infty(\cM,\tau))$.
\end{theorem}
\begin{proof}
(i) It follows from Proposition \ref{prop:6.1} that for every $X\in \orb(A;L_0(\cM,\tau),L_\infty(\cM,\tau))$, there exists such a $C>0$ satisfying \eqref{CCC}.

(ii) Conversely, assume that $A,X\in S(\cM,\tau)$ satisfies \eqref{CCC}.

Then, by \cite[Lemma 1.3]{CKS}, there are isomorphisms  $J_A$ between $S(M_{\mu(A)},m)$ and $S(\cM_{|A|},\tau)$ such that $J_A \mu(A)=|A|$  and $J_X$ between $S(M_{\mu(X)},m)$ and $S(\cM_{|X|},\tau)$ such that $J_X\mu(X)=|X|$.

Then, by \eqref{CCC},
 we have
$$ 2C \mu(\frac{t}{2C};A )-\mu(t;x)  \ge C \mu(\frac{t}{2C};A ) + (C \mu(\frac{t}{2C};A )-\mu(t;x))  \ge C \mu(\frac{\tau(s(X))}{2C};A )$$
for every $t\in (0, \tau(s(X)))$.
Without loss of generality, we can assume that $C$ is an  integer which is large enough such that
\begin{align}\label{ineq2C}
2C \mu(\frac{t}{2C};A )-\mu(t;X)\ge C \mu(\frac{\tau(s(X))}{2C};A ) \ge 1
\end{align}
for every $t\in (0, \tau(s(X)))$.

Let $\varepsilon<\frac{1}{2C}$.
Since $J_A \mu(A) = |A|$, we can define $A_n:= [a_n,b_n)$, $n\ge 0$, by
$$\chi_{A_n}= \chi_{ [a_n,b_n)}=J_A^{-1} (E^{|A|} (n\varepsilon,(n+1)\varepsilon]).$$
For every $ 0\le j\le 2C-1$, we set
$$[a_{nj},b_{nj}) := [2C a_n  +  j(b_n -a_n ), 2C a_n  +  (j+1)(b_n -a_n )). $$

By \cite[Lemma 1.3]{CKS},  for every  $ [a_{nj},b_{nj})\subset [0,\tau(s(X)))$, we have
\begin{align}\label{16}
\tau(E^{|A|} (n\varepsilon,(n+1)\varepsilon])=\tau(J_A \chi_{[a_n,b_n)}) = \tau(J_X  \chi_{[a_{nj},b_{nj})}).
\end{align}
Since $\cM$ is a finite factor, due to \eqref{16}, there exist  partial isometries $U_{nj} \in \cM$ such that
$U_{nj}U_{nj}^* = E^{|A|}(n\varepsilon,(n+1)\varepsilon]=J_A \chi_{[a_n,b_n)} $
and
$U_{nj}^*U_{nj} = J_X  \chi_{[a_{nj},b_{nj})} $.
If $ [a_{nj},b_{nj}) \cap [0,\tau(s(X)))=\varnothing$, we define $U_{nj}:=0$.
In the case when $ [a_{nj},b_{nj})\nsubseteq [0,\tau(s(X)))$ but $ [a_{nj},b_{nj}) \cap [0,\tau(s(X)))\ne \varnothing $,
we define $U_{nj}$ as  the partial isometry such that
$U_{nj}U_{nj}^* =J_A \chi_{[a_n, a_n -a_{nj} + \tau(s(X))) }\le E^{|A|}(n\varepsilon,(n+1)\varepsilon] $
and
$U_{nj}^*U_{nj} = J_X  \chi_{[a_{nj},\tau(s(X)))} $.

Denote $U_j :=\sum_{n=0}^\infty U_{nj}$, $0\le j\le 2C-1$ (note that $U_{kj}^* U_{lj}=0$ and $U_{lj} U_{kj} ^* =0$ whenever $k\ne l$).
Note that every $U_j$ is a partial isometry.
Let
$$B_1:=\sum_{n=1}^{\infty} n\varepsilon E^{|A|}(n\varepsilon,(n+1)\varepsilon]=\sum_{n=1}^{\infty} n\varepsilon J_A \chi_{[a_n,b_n)} \in S(\cM_{|A|},\tau) $$
and
$$B_2:= \sum_{n,j}n\varepsilon J_X  ( \chi_{[a_{nj},b_{nj})}\chi_{[0,\tau(s(X)))} )\in S(\cM_{|X|},\tau).$$
Then, noting that $U_{n,j}^*  E^{|A|}(n\varepsilon,(n+1)\varepsilon]U_{n,j}= U_{n,j}^* J_A \chi_{[a_n,b_n)} U_{n,j} = J_X  ( \chi_{[a_{nj},b_{nj})}\chi_{[0,\tau(s(X)))} )$ for every $n,j$, we have
$$
B_ 2 =  \sum_{n,j}  U_{nj} ^*  n\varepsilon E^{|A|}(n\varepsilon,(n+1)\varepsilon] U _{nj}= \sum_{j=0}^{2C-1} U_{j} ^* B_1 U _{j}.
$$
Since $\varepsilon<\frac{1}{2C}$, it follows that $0 \le  \mu(A) - \mu(B_1 ) <\frac{1}{2C}$ and therefore, by \eqref{ineq2C}, we have
\begin{align}\label{17}
2C  \mu(t;B_2 )  -\mu(t;X) &=2C  \mu(\frac{t}{2C};B_1 )  -\mu(t;X) \nonumber\\
&>    2C (\mu(\frac{t}{2C};A ) -\frac{1}{2C})-\mu(t;X)   \ge 0
\end{align}
for every $t\in [ 0,  \tau(s(X)))$ (note that the $  \mu(t;B_2 )  =\mu(\frac{t}{2C};B_1 )$ follows immediately from the definitions of $B_1$ and $B_2$).

Let $A_\Delta:= \int  \frac{[\lambda/\varepsilon] \varepsilon }{\lambda}  dE ^{|A|}(\lambda)$
and $B_\Delta:= J_X \frac{\mu(X)}{\mu(B_2)}$.
Here, $[\lambda/\varepsilon]$ is the  integer part of $\lambda/\varepsilon$.
Note that \eqref{17} implies that  $B_\Delta$ is a bounded operator.
Clearly, we have $A_\Delta |A| = B_1$ and $|X| =J_X \frac{\mu(X)}{\mu(B_2)} J_X \mu(B_2)= B_\Delta B_2 $ (notice that $\sum_{n,j}n\varepsilon  \chi_{[a_{nj},b_{nj})}\chi_{[0,\tau(s(X)))} =\mu(B_2)$).
Let $U_A |A| = A$ and $U_X|X| =X$ be the polar decompositions.
Define a homomorphism $T :S(\cM,\tau) \rightarrow S(\cM,\tau) $ by setting
$$ T Z =U_X  B_\Delta \Big(  \sum_{j=0}^{2C-1}U_j^*( A_\Delta  U_A^* Z )U_j \Big)  , ~ Z\in S(\cM,\tau). $$
It is easy to verify that $TA=X$
 and $T$ is a bounded homomorphism  on the pair  $(L_0(\cM,\tau),\cM)$ (one should note that operators of multiplication by  $U_j$, $0\le j\le 2C-1$, $A_\Delta$ and $B_\Delta$ are  bounded homomorphisms on the pair  $(L_0(\cM,\tau),\cM)$  by Corollary \ref{cor:inter}).

The last statement follows immediately  from Proposition \ref{prop:6.5}.
\end{proof}

\begin{remark}
The assumption that $\cM$ is a finite factor plays a crucial role in the above proof.
The authors did not succeed in extending the result to the case for general semifinite von Neumann algebras.
\end{remark}
{\bf Acknowledgements}
The authors would like to thank Sergei Astashkin and Dima Zanin for  helpful discussions.

The first author acknowledges the support of University International Postgraduate Award (UIPA).
The second author was supported by the Australian Research Council.



\begin{thebibliography}{04} \frenchspacing
\bibitem{AK}
M. Acosta, A. Kami\'{n}ska,
{\it  Weak neighborhoods and the Daugavet property of the interpolation spaces $L^1+L^\infty$ and $L^1 \cap L^\infty$,}
Indiana Univ. Math. J. 57 (1) (2008), 77--96.

\bibitem{Astashkin}
S.V. Astashkin,
{\it Interpolation of operators in quasinormed groups of measurable functions,}
Siberian Math. J. 35 (6) (1994), 1075--1082.

\bibitem{AM}
S.V. Astashkin, L. Maligranda,
{\it Interpolation between $L_1$ and $L_p$, $1<p<\infty$},
Proc. Amer. Math. Soc. 132 (10) (2004), 2929--2938.





\bibitem{Bergh_L}
J. Bergh, J. L\"{o}fstr\"{o}m,
{\it Interpolation spaces, an introduction,}
Springer-Verlag, Berlin, Heidelberg, and New York, 1976.



\bibitem{Calderon} A. Calder\'{o}n, {\it Spaces between $L^1$ and $L^\infty$ and the theorem of Marcinkiewicz,} Studia Math., 26 (1966), 273--299.



\bibitem{CKS} V. Chilin, A. Krygin, F. Sukochev, {\it Local uniform and uniform convexity of non-commutative symmetric spaces of measurable operators,} Math. Proc. Camb. Phil. Soc. 111 (1992), 355--368.








\bibitem{Dixmier}
J. Dixmier,
{\it Les algebres d'operateurs dans l'Espace Hilbertien, 2nd ed.,}
Gauthier-Vallars, Paris, 1969.


\bibitem{DDP1992} P. Dodds, T. Dodds, B. de Pagter, {\it Fully symmetric operator spaces,} Integr. Equ. Oper. Theory 15 (1992), 942--972.

\bibitem{DDP} P. Dodds, T. Dodds, B. de Pagter, {\it Noncommutative K\"{o}the duality,} Trans.  Amer. Math. Soc. 339 (2) (1993), 717--750.

\bibitem{DDS2014} P. Dodds, T. Dodds, F. Sukochev, {\it  On $p$-convexity and $q$-concavity in non-commutative symmetric spaces, } Integr. Equ. Oper. Theory  78 (2014), 91--114.



\bibitem{DP2} P. Dodds, B. de Pagter, {\it Normed K\"{o}the spaces: A non-commutative viewpoint,} Indag. Math. 25 (2014), 206--249.


\bibitem{DPS}
P. Dodds, B. de Pagter, F. Sukochev,
{\it Theory of noncommutative integration,} unpublished manuscript.




\bibitem{DSZ}
K. Dykema, F. Sukochev, D. Zanin,
{\it An upper triangular decomposition theorem  for some unbounded operators affiliated to $II_1$-factors,}
Israel J. Math. 222 (2) (2017) 645--709.




\bibitem{FK}
T. Fack, H. Kosaki,
{\it Generalized $s$-numbers of $\tau$-measurable operators,}
Pacific J. Math. 123 (2) (1986), 269--300.

\bibitem{GK1} I. Gohberg, M. Krein, {\it Introduction to the theory of linear nonselfadjoint operators,}
Translations of Mathematical Monographs, Vol. 18 American Mathematical Society, Providence, R.I. 1969.


\bibitem{GK2} I. Gohberg, M. Krein, {\it Theory and applications of Volterra operators in Hilbert space,}
Translations of Mathematical Monographs, Vol. 24 American Mathematical Society, Providence, R.I. 1970.


\bibitem{HLS}
J. Huang, G. Levitina, F. Sukochev,
{\it Completeness of symmetric $\Delta$-normed spaces of $\tau$-measurable operators,}
Studia Math. 237 (3) (2017), 201--219.

\bibitem{HM}
H. Hudzik, L. Maligranda,
{\it An interpolation theorem in symmetric function $F$-spaces,}
Proc. Amer. Math. Soc. 110 (1) (1990), 89--96.

\bibitem{KR1}
R. Kadison, J. Ringrose,
{\it Fundamentals of the Theory of Operator Algebras I},
Academic Press, Orlando, 1983.

\bibitem{KR2}
R. Kadison, J. Ringrose,
{\it Fundamentals of the Theory of Operator Algebras II},
Academic Press, Orlando, 1986.


\bibitem{KPR}
N. Kalton, N. Peck, J. Rogers,
{\it An F-space Sampler},
 London Math. Soc. Lecture Note Ser., vol.89, Cambridge University Press, Cambridge, 1985.


\bibitem{Kalton_S}  N. Kalton, F. Sukochev,
{\it Symmetric norms and spaces of operators},
 J. Reine Angew. Math.  621 (2008), 81--121.



\bibitem{KPS} S. Krein, Yu. Petunin, E. Semenov, {\it Interpolation of linear operators.} Translated from  Russian by J. Sz\~ucs. Translations of Mathematical Monographs, 54. American Mathematical Society, Providence, R.I., 1982.


\bibitem{LSZ}
S. Lord, F. Sukochev, D. Zanin,
{\it Singular traces: Theory and applications,}
De gruyter studies in Mathematical Physics 46 (2012).


\bibitem{Mali}
L. Maligranda,
{\it The $\cK$-functional for symmetric spaces},
Lecture Notes in Math. 1070 (1984), 169--182.

\bibitem{Mali1992}
L. Maligranda, V. Ovchinnikov,
{\it On interpolation between $L_1+L_\infty$ and $L_1\cap L_\infty$,}
J. Funct. Anal. 107 (1992), 342--351.

\bibitem{MP1991} L. Maligranda, L. Persson,
{\it The $E$-functional for some pairs of groups,} Results in Math. 20 (1991), 538--553.


\bibitem{Mityagin}
B. Mitjagin,
{\it An interpolation theorem for modular spaces,}
Mat. Sb. 66 (108) (1965), 473--482.

\bibitem{Muratov} M. Muratov, V. Chilin, {\it Algebras of measurable and locally measurable operators,} Kyiv, Pratsi In-ty matematiki NAN Ukraini 69 (2007) (in Russian).


\bibitem{Nelson}
E. Nelson,
{\it Notes on non-commutative integration,}
J. Funct. Anal.  15 (1974), 103-116.








\bibitem{Rotfeld}
S. Rotfel'd,
{\it Analogues of the interpolation theorems of Mitjagin and Semenov for operators in non-normed symmetric spaces,}
 Problems of mathematical analysis,  Izdat. Leningrad. Univ., Leningrad, 1973.


\bibitem{Se}
I. Segal,
{\it A non-commutative extension of abstract integration},
Ann. Math. 57 (1953), 401--457.


\bibitem{Simon} B. Simon, {\it Trace ideals and their applications,} Second edition. Mathematical Surveys and Monographs,  120. American Mathematical Society, Providence, RI, 2005.
%

\bibitem{Sukochev} F. Sukochev, {\it Completeness of quasi-normed symmetric operator spaces,} Indag. Math. 25 (2014), 376--388.

\bibitem{SCh_90}
F. Sukochev, V. Chilin,
{\it Symmetric spaces over semifinite von Neumann algebras,}
Dokl. Akad. Nauk SSSR  313  (1990),  no. 4, 811--815 (Russian). English translation:
Soviet Math. Dokl. 42 (1991), 97--101.

\bibitem{STZ}
F. Sukochev, K. Tulenov, D. Zanin,
{\it Nehari-type Theorem for non-commutative Hardy spaces,}
J. Geom. Anal 27 (2017), 1789--1802.

\bibitem{Tak}
M. Takesaki,
{\it Theory of Operator Algebras I},
 Springer-Verlag, New York, 1979.

\bibitem{Ta2}
M. Takesaki,
{\it Theory of Operator Algebras II},
 Springer-Verlag, Berlin-Heidelberg-New York, 2003.




\end{thebibliography}
\end{document}